\documentclass{amsart}
\usepackage{enumerate,graphicx}
\usepackage{verbatim}

\usepackage{amsmath, amssymb, amsthm}
\usepackage{amsfonts}

\newtheorem{theorem}{Theorem}
\newtheorem{proposition}[theorem]{Proposition}
\newtheorem{lemma}[theorem]{Lemma}
\newtheorem{corollary}[theorem]{Corollary}

\theoremstyle{definition}

\newtheorem{problem}[theorem]{Problem}
\newtheorem*{acknowledgements}{Acknowledgements}
\newtheorem*{conjecture}{Conjecture}

\numberwithin{equation}{section}

\newcommand\vanish[1]{}	

\newcommand\ourcomment[1]{ \textbf{[#1]} }
\newcommand\oc\ourcomment

\begin{document}

\title{Some problems on induced subgraphs}
\author{Vaidy Sivaraman}

\address{Department of Mathematical Sciences, Binghamton University.}
\email{vaidy@math.binghamton.edu}

\begin{abstract}

We discuss some problems related to induced subgraphs. 
The first problem is about getting a good upper bound for the chromatic number in terms of the clique number for graphs in which every induced cycle has length $3$ or $4$. 
The second problem is about the perfect chromatic number of a graph, which is the smallest number of perfect sets into which the vertex set of a graph can be partitioned. (A set of vertices is said to be perfect it it induces a perfect graph.) 
The third problem is on antichains in the induced subgraph ordering. The fourth problem is on graphs in which the difference between the chromatic number and the clique number is at most one for every induced subgraph of the graph.
The fifth problem is on a weakening of the notorious Erd\H{o}s-Hajnal conjecture. 
The last problem is on a conjecture of Gy\'{a}rf\'{a}s about $\chi$-boundedness of a particular class of graphs. 
\end{abstract}



\date{\today}

\maketitle




\date{\today}

\maketitle

\begin{section}
{Introduction}
All graphs considered in this article are finite and simple.  Let $G$ be a graph. A graph that can be obtained from $G$ by deleting some of its vertices is called an induced subgraph of $G$. A hole in a graph is an induced cycle of length at least $4$. 
There are several notions of containment for graphs, and induced subgraph is the strongest one.  A class of graphs is said to be hereditary if every induced subgraph of every graph in the class is also in the class.  The chromatic number of a graph $G$ is denoted by $\chi(G)$ and the clique number by $\omega(G)$. 
\end{section}
\begin{section}
{Graphs in which every induced cycle is a triangle or a square}

A graph in which every induced cycle is a triangle is called a chordal graph. (There are several other names, like triangulated graph, rigid circuit graph, and perfect elimination graph.) 
A graph in which every induced cycle is a square is called a chordal bipartite graph. 
Both classes are perfect: Chordal graphs was the first interesting class proved to be perfect in the late 1950s, and chordal bipartite graphs are bipartite, and hence perfect. 
What if every cycle in a graph is either a triangle or a square. Such graphs need not be perfect.  The complement of a $7$-cycle is an example. 

We prove that the class of graphs not containing holes of length at least $5$ is $\chi$-bounded by the function $f(x) = 2^{2^x}$. The proof uses the levelling argument and several ideas from the recent paper by Scott and Seymour \cite{SS1}. 
The main point here is that since we are forbidding all holes of length at least $5$, we can bypass their ``spine", ``parent rule", and ``parity property". Also, it gives a slightly better bound, although not in the final $\chi$-bounding function.

A levelling $L$ is a sequence of disjoint subsets of vertices $(L_0, L_1, \cdots, L_k)$ such that $|L_0|=1$ and every vertex in $L_i$ has a neighbor in $L_{i-1}$, and no vertex in $L_i$ has a neighbor in $L_j$ for $j < i-1$.

\begin{lemma}\label{MAINLEMMA}
Suppose that every graph with no holes of length at least $5$ and clique number at most $\omega-1$ has chromatic number at most $n$. Let $G$ be a graph with no holes of length at least $5$ and clique number $\omega$. Then $\chi(G) \leq 4n^2$. 
\end{lemma}

\begin{proof} 

We may assume that $G$ is connected. Let $v \in V(G)$. Let $L_i$ be the set of vertices at distance $i$ from $v$.  The proof is by analyzing this levelling, in particular, looking at a level and previous two levels. We show that for every $k$, $\chi(L_k) \leq 2n^2$. This is trivially true for $k=0$. Consider $k=1$. In fact, $L_1$ is the set of neighbors of the vertex in $L_0$, and hence has clique number at most $\omega-1$, and hence chromatic number at most $n$. Now let $k \geq 2$. We will show how to color the vertices in $L_k$ with $2n^2$ colors. Since we can use the same set of colors for each component of $L_k$, we may assume that $L_k$ is connected. By deleting vertices in $L_0, L_1, \cdots ,L_{k-1}$  that don't have a child for which it is the only parent, we may assume every vertex in $L_0, L_1, \cdots, L_{k-1}$ has a child for which it is the only parent. Let $x \in L_{k-2}$. Let $y$ be a child of $x$ such that $y$ is its only parent. We partition $L_{k-1} - \{y\}$ into $A,B$ where $A$ is the set of vertices in $L_{k-1}$ that are neighbors of $y$, and $B = L_{k-1} - A - \{y\}$. Note that $A$ has clique number at most $\omega-1$, and hence chromatic number at most $n$. Suppose there is a vertex  $z \in B$ that is not adjacent to $x$. Let $a$ be a parent of $z$. Now there is a path between $a$ and $x$ with interior in $L_0 \cup L_1 \cup \cdots \cup L_{k-2}$, let $P$ be a shortest such path. Also, there is a path between $y$ and $z$ with interior in $L_k $ (this is guaranteed by the connectedness of $L_k$), let $P'$ be a shortest such path. Now $a-P-x-y-P'-z-a$ is a hole of length at least $5$, which is impossible. Hence we conclude that  $x$ is adjacent to every vertex in $B$.  Hence $B \cup \{y\}$ has clique number at most $\omega-1$, and chromatic number at most $n$. By using different colors for $A$ and $B \cup \{y\}$, we see that $L_{k-1}$ can be colored with at most $2n$ colors. Now partition $L_k$ into sets $A_1, \cdots, A_{2n}$ as follows:   A vertex is in $A_i$ if $i$ is the smallest of the colors of its neighbors in $L_{k-1}$. 

We will show that each $A_i$ has clique number at most $\omega-1$. Let $K$ be a clique in some $A_i$. Suppose there exist $u,v \in K$ such that each has a parent of color $i$ that is not a parent of the other (say $u'$ is a parent of $u$ but not $v$, $v'$ is a parent of $v$ but not $u$. Note that $uv$ is an edge (they belong to a clique) and $u'v'$ is a non-edge (both $u'$ and $v'$ received the same color in a proper coloring). Let $P$ be a shortest path between $u'$ and $v'$ with interior vertices in $L_0 \cup \cdots \cup L_{k-2}$. Now $u-v-v'-P-u'-u$ is a hole of length at least $5$. Hence for any two vertices $u,v \in K$ every parent of $u$ of color $i$ is also a parent of $v$ (or vice versa). Thus the set of parents of color $i$ of vertices in $K$ form a chain, and hence there must be a vertex of color $i$ adjacent to every vertex in $K$. Hence $|K| \leq \omega - 1$. Thus each $A_i$ has clique number at most $\omega-1$, and hence chromatic number at most $n$. By using different set of colors for different $A_i$, we conclude that $\chi(L_k) \leq 2n^2$. By using one set of colors for odd levels and another set for even levels, we conclude that $\chi(G) \leq 4n^2$.
\end{proof}

\begin{theorem}
Let $G$ be a graph with no holes of length at least $5$. Then $\chi(G) \leq 2^{2^{\omega(G)}}$. 
\end{theorem}

\begin{proof} We claim that if $G$ is a graph with no holes of length at least $5$, then $\chi(G) \leq (\frac{1}{4}) 2^{2^{\omega(G)}}$. The proof is by induction on $\omega(G)$. The base case $\omega = 1$ is trivial. Suppose the statement is true for some $k$ i.e., every graph with no holes of length at least $5$ and clique number $k$ has chromatic number at most $(\frac{1}{4}) 2^{2^k}$. Let $G$ be a graph with no holes of length at least $5$ and clique number $k+1$. By Lemma \ref{MAINLEMMA}, $\chi(G) \leq 4((\frac{1}{4}) 2^{2^k})^2 = (\frac{1}{4}) 2^{2^{k+1}}$. This completes the induction step. 
\end{proof}

Scott and Seymour \cite{SS1} proved almost the same bound for a much bigger class, viz. graphs not containing odd holes. The proof here mimics their proof but is much easier because of the stronger hypothesis. Note that the graphs mentioned in \cite{SS1} have no holes of size at least $5$ and have $\chi > \omega^{\frac{log 3.5}{log 3}}$. Their example is as follows: Let $G_0$ have one vertex, and for $k \geq 1$ let $G_k$ be obtained from $G_{k-1}$ by substituting a seven-vertex antihole for each vertex. Then $G_k$ has no hole of length at least $5$, $\omega(G_k) = 3^k$, and $\chi(G_k) \geq (\frac{7}{2})^k$.  We conjecture the following.

\begin{conjecture}
Let $G$ be a graph with no holes of length at least $5$. Then $V(G)$ can be partitioned into two sets, none of them containing a maximum clique of $G$. 
\end{conjecture}

Hoang and McDiarmid (\cite{HM}) conjecture that the previous statement actually holds for all odd-hole-free graphs. The truth of the above conjecture will immediately imply $\chi(G) \leq 2^{\omega(G)}$ for a graph $G$ with no holes of length at least $5$. It is possible that the following stronger conclusion holds.

\begin{conjecture}
Let $G$ be a graph with no holes of length at least $5$. Then $\chi(G) \leq \omega(G)^2$. 
\end{conjecture}

Paul Seymour (private communication) has recently proved that $\chi(G) \leq 2^{{\omega(G)}^2}$ holds for every graph $G$ with no holes of length at least $5$.  \\\

Here is a numerical problem. 

\begin{problem}
Let $f(n)$ be the largest chromatic number of a graph with no holes of length at least $5$ and clique number $n$. Clearly $f(1) = 1$ and $f(2) = 2$. It is known from \cite{CRST} that $f(3)=4$. Determine $f(4)$.  
\end{problem}

Chordal graphs (see \cite{GAD}), have fantastic properties, like the following:

\begin{itemize}
\item Has a simplicial ordering 
\item Every minimal cutset is a clique
\item Intersection of subtrees of a tree
\end{itemize}

We believe that there should be a characterization based on a relaxation of one of these concepts. 

\begin{problem}
Give a structural characterization of graphs with no holes of length at least $5$. 
\end{problem}

Recognizing and optimizing graphs with no holes of length at least $5$ also look interesting. 

\begin{problem}
Give a polynomial time algorithm  to recognize graphs with no holes of length at least $5$. 
\end{problem}

\begin{problem}
Can we  determine the clique number, chromatic number, stability number, and clique cover number of a graph not containing  holes of length at least $5$ in polynomial time.
\end{problem}

Note that if we insist that we don't have holes of length at least $5$ both in the graph and its complement, we have what is called a ``weakly chordal graph". Ryan Hayward (\cite{RBH}) introduced these graphs and proved a structure theorem for them that easily implied that they are perfect.

We end this section with two more problems and a simple proposition.

\begin{problem}
Determine a good $\chi$-bounding function for the class of graphs in which every induced cycle is either a triangle, square or a pentagon.
\end{problem}

\begin{problem}
Determine a good $\chi$-bounding function for the class of graphs in which every induced cycle is a pentagon.
\end{problem}

\begin{proposition}
Let $\mathcal{G}$ be the class of graphs in which any two induced cycles have the same parity. Then $\mathcal{G}$ is $\chi$-bounded.
\end{proposition}

\begin{proof}
If the parity is even, then the graphs are bipartite. 
If the parity is odd, then the graphs are even-hole-free, and the main result of \cite{ACHRS} is that such graphs contain a bisimplicial vertex (a vertex whose neighborhood is a union of two cliques), and hence satisfy $\chi \leq 2\omega - 1$. 
\end{proof}

\end{section}

\begin{section}
{Perfect chromatic number}
In this section, we show that it is NP-complete to determine whether the vertex set of a given graph can be partitioned into two sets such that the graph induced by each set is perfect.



The perfect chromatic number, denoted by $\chi_p$, of a graph is the smallest number of colors required to color the vertices of a graph so that each color  class induces a perfect graph. 
Graphs with $\chi_p=1$ are precisely the perfect graphs. A polynomial-time algorithm is known for recognizing perfect graphs \cite{CCLSV}, and hence there is a polynomial-time 
algorithm to determine whether a given graph has $\chi_p=1$. It turns out that determining whether a given graph has $\chi_p=2$ is NP-complete.

\begin{lemma}\label{BOUNDS}
Let $G$ be a graph. Then $\frac{\chi(G)}{\omega(G)} \leq \chi_p(G) \leq \lceil \frac{\chi(G)}{2} \rceil$. 
\end{lemma}
\begin{proof}
Let $\{C_1, \cdots, C_{\chi_p(G)}\}$ be a partition of $V(G)$ into perfect sets. Since $G:C_i$ is perfect, $\chi(C_i) \leq \omega(G)$. 
Giving different sets of colors to $C_1, \cdots, C_{\chi_p(G)}$, we get a $\chi_p(G)\omega(G)$-coloring of $G$. Hence $\chi(G) \leq \chi_p(G)\omega(G)$, establishing the first inequality. 
Let $\{C_1, \cdots, C_{\chi(G)}\}$ be a partition of $V(G)$ into stable sets.  The union of two color classes induces a bipartite graph, and hence a perfect graph. So we can pair the $C_i$ into perfect sets. 
Thus $\chi_p(G) \leq \lceil \frac{\chi(G)}{2} \rceil$. 
\end{proof}
\begin{lemma}\label{HALVINGLEMMA}
Let $G$ be a triangle-free graph. Then $\chi_p(G) = \lceil \frac{\chi(G)}{2} \rceil$. 
\end{lemma}

\begin{proof}
We may assume that $G$ has edges, for the desired relation is trivially true for edgeless graphs. Since $G$ is triangle-free,  $\omega(G) = 2$. The desired result follows from Lemma \ref{BOUNDS}. 
\end{proof}

\begin{theorem}
The problem of determining whether $\chi_p=2$ for a given graph is NP-complete.
\end{theorem}

\begin{proof}
We show that the restricted problem of determining whether a given triangle-free graph has $\chi_p=2$ is NP-complete. 
Let $G$ be a given triangle-free graph and we would like to find out whether $\chi_p(G) = 2$ or not. By Lemma \ref{HALVINGLEMMA}, $\chi_p(G) = 2$ if and only if $\chi(G)=3$ or $4$. 
Since the problem of determining whether a given triangle-free graph is $4$-colorable was shown to be NP-complete by Maffray and Preissmann \cite{MP}, we conclude that it is NP-complete to determine whether a given triangle-free graph has $\chi_p \leq 2$. 
Since it is easy to determine whether a graph has $\chi_p = 1$, we are done. 
\end{proof}

A natural question is to ask whether a $\chi$-bounded class of graphs is $\chi_p$ bounded, i.e., is there an absolute bound on $\chi_p$ for graphs in the class. For perfect graphs, $\chi_p =1$.
 A first step would be to consider graphs in which every induced subgraph has $\chi - \omega \leq 1$. We call such graphs nice. 
\begin{theorem}
Nice graphs have unbounded $\chi_p$. 
\end{theorem}

\begin{proof}
By Vizing's theorem, line graphs are nice, and therefore it suffices to prove that line graphs are not $\chi_p$ bounded. Consider the line graph of $K_n$. 
A partition of the vertex set of $L(K_n)$ into $k$ perfect sets gives a partition of the edge set of $K_n$ into $k$ sets such that no one contains a $5$-cycle (as a subgraph). 
But if $n  >  R(5, \cdots ,5)$, where $R(5, \cdots ,5)$ is the Ramsey number for $k$ colors, then this cannot happen. 
Thus, if $n  >  R(5, \cdots ,5)$, $\chi_p(L(K_n)) > k$. 
\end{proof}

A hereditary class of graphs $\mathcal{G}$ is said to be $\chi_p$-bounded if there exists a function $f$ such that every graph $G \in \mathcal{G}$ satisfies $\chi_p(G) \leq f(\omega(G))$. 
\begin{lemma}
A graph class if $\chi$-bounded if and only if it is $\chi_p$-bounded.  
\end{lemma}
\begin{proof}
Follows easily from Lemma \ref{BOUNDS}. 
\end{proof}

\begin{problem}
Characterize graphs with $\chi_p = 2$.
\end{problem}

\begin{problem}
Determine $\chi_p(L(K_n))$.
\end{problem}

By Lemma \ref{BOUNDS} and the $4$-color theorem, we see that $\chi_p$ of a planar graph is at most $2$. Is there a way to prove this directly without resorting to the $4$-color theorem?

\begin{problem}
Find a direct proof as to why the vertex set of a planar graph can be partitioned into two perfect sets. 
\end{problem}

\end{section}

\begin{section}
{Antichains in the induced subgraph ordering}

This section is motivated by the celebrated Robertson-Seymour graph minor theorem. 

\begin{theorem}[Robertson-Seymour theorem]
Every antichain in the minor ordering is finite. 
\end{theorem}

In the induced subgraph ordering, there can be infinite antichains. For example, $C_3, C_4, \cdots$ is a standard example. 
Another one: Let $T_k$ be obtained by adding two pendant edges at each point of the path with $k$ edges. Then $T_1, T_2, \cdots$ is an antichain.

In this section, we show that there exist real numbers $a,b > 0$ such that every sequence $f$ with $f_n \leq ae^{bn}$ is a forbidden sequence.

The forbidden sequence of a hereditary family of graphs $\mathcal{G}$ is the sequence whose $n$th term is the number of $n$-vertex minimal graphs not in $\mathcal{G}$.  A sequence $f$ is called a forbidden sequence if there is a hereditary graph class whose forbidden sequence is $f$.

\begin{theorem}
There exist real numbers $a,b > 0$ such that every sequence $f$ with $f_n \leq ae^{bn}$ is a forbidden sequence.
\end{theorem}

\begin{proof}
The set of connected $4$-regular graphs is an antichain. To see this, let $G,H$ be non-isomorphic connected $4$-regular graphs. Suppose $H$ is an induced subgraph of $G$.  We may assume $|V(H)| < |V(G)|$.  
Since $H$ is an induced subgraph of $G$ and since $H$ is connected, $H$ is a component of $G$. This is a contradiction since $G$ is connected and has more vertices than $H$. 

The asymptotics of the number of connected $4$-regular graphs is known and the theorem follows from that (see \cite{BDM}, \cite{NCW}). We can choose any set $S$ of connected $4$-regular graphs and the class of graphs not containing 
any graph in $S$ is certainly heredittary, and has forbidden sequence $S$. So all we have to do is to choose $f_n$ connected $4$-regular graphs on $n$ vertices. As long are there are at least $f_n$ connected 
$4$-regular graph on $n$ vertices, this will be possible, and hence $f_n \leq ae^{bn}$ suffices. 
\end{proof}

\end{section}

\begin{section}
{Nice graphs}

A graph is perfect if  for every induced subgraph $H$ of $G$, $\chi(H) - \omega(H) = 0$.  We relax this definition to get a bigger class of graphs. A graph G is said to be nice if for every induced subgraph $H$ of $G$, $\chi(H) - \omega(H) \in \{0,1\}$. 
We study this new class of graphs.

\begin{proposition}\label{PERFECT}
Every perfect graph is nice.
\end{proposition}

\begin{proposition}\label{TRIPARTITE}
Every tripartite graph is nice.
\end{proposition}

\begin{proposition}\label{PLANAR}
Every planar graph is nice.
\end{proposition}

\begin{proof}
Let $G$ be a planar graph. If $G$ has chromatic number at most $3$, we are done by Proposition \ref{TRIPARTITE}. If not, we know by the Four Color Theorem that G has chromatic number $4$. By Gr\"{o}tzsch's theorem, $G$ 
must have a triangle, and hence $\omega(G) \in \{3,4\}$. So we have $\chi(G) - \omega(G) \in \{0,1\}$, and hence $G$ is nice. 
\end{proof}

\begin{proposition}\label{LINEGRAPH}
Every line graph is nice.
\end{proposition}
\begin{proof}
This follows from Vizing's Theorem. 
\end{proof}

\begin{proposition}\label{TRIANGLE-FREE}
A triangle-free graph with chromatic number at least $4$ is not nice. 
\end{proposition}

\begin{proof}
Let $G$ be a triangle-free graph with $\chi(G) \geq 4$. Since $G$ is triangle-free $\omega(G) \leq 2$. Hence $\chi(G) - \omega(G) \geq 2$, and so $G$ is not nice. 
\end{proof}

\begin{corollary}
A triangle-free graph is nice if and only if its chromatic number is at most $3$. 
\end{corollary}

\begin{proof}
By Propositions \ref{TRIPARTITE} and \ref{TRIANGLE-FREE}. 
\end{proof}

What properties of perfect graphs are also enjoyed by nice graphs? Lov\'{a}sz' perfect graph theorem that complementation preserves perfection is blatantly false for nice graphs. 

\begin{proposition}
The quantity $\chi - \omega$ can be arbitrarily large for the complement of a nice graph. 
\end{proposition}

\begin{proof}
The line graph of $K_n$ ($n \geq 7$ ) is nice by Proposition \ref{LINEGRAPH}. But $\chi(\overline{L(K_n)}) = n-2$ and $\omega(\overline{L(K_n)}) = \lfloor \frac{n}{2} \rfloor$. 
\end{proof}

\begin{corollary}
The class of nice graphs is not closed under complementation.
\end{corollary}

Determining the chromatic number, clique number, stability number, and clique cover number of a perfect graph can be done in polynomial time. 

\begin{proposition}
The problem of determining the stability number and chromatic number of a nice graph are both NP-complete.
\end{proposition}
\begin{proof}
The two problems are both NP-complete already for planar graphs. 
\end{proof}

\begin{problem}
Is there a polynomial-time algorithm to determine the clique number of a nice graph?
\end{problem}

\begin{problem}
Is there a polynomial-time algorithm to determine the clique cover number of a nice graph?
\end{problem}

\begin{problem}
(Forbidden Induced Subgraph Characterization) Determine the set of minimal non-nice graphs. 
\end{problem}

\begin{problem}
Can nice graphs be recognized in polynomial time? 
\end{problem}

\begin{problem}
Characterize claw-free nice graphs. 
\end{problem}

\end{section}

\begin{section}
{Weakening the Erd\H{o}s-Hajnal conjecture}

A famous open problem concerning induced subgraphs is the Erd\H{o}s-Hajnal conjecture. 

\begin{conjecture}[Erd\H{o}s-Hajnal 1989]
For every graph $H$, there exists $\epsilon(H) > 0$ such that every graph $G$ not containing $H$ as an induced subgraph has either a stable set or a clique of size at least $|V(G)|^{\epsilon(H)}$. 
\end{conjecture}

Chudnovsky and Seymour proposed a weakening of this by excluding both a graph and its complement. We propose a different weakening.

\begin{conjecture}
For every graph $H$, there exists $\epsilon(H) > 0$ such that every graph $G$ not containing a subdivision of $H$ as an induced subgraph has either a stable set or a clique of size at least $|V(G)|^{\epsilon(H)}$. 

\end{conjecture}

What about when $H$ is the $5$-cycle? We pose this as a problem.

\begin{problem}
 Find an $\epsilon > 0$, if one exists, such that every nice graph $G$ has either a stable set or a clique of size at least $|V(G)|^{\epsilon}$.

\end{problem}
\end{section}

\begin{section}
{A conjecture of Gy\'{a}rf\'{a}s}
Gy\'{a}rf\'{a}s asked in 1985 whether the class of graphs whose every induced subgraph satisfies $\alpha \omega \geq n-1$ is $\chi$-bounded. 
We observe that the $\chi$-boundedness of such a family with any constant instead of 1 follows from a recent result of Scott and Seymour.

Let $\mathcal{G}$ be the class of graphs whose every induced subgraph satisfies $\alpha \omega \geq n-1$. Gy\'{a}rf\'{a}s (\cite{AG}, Problem 6.8) asked whether $\mathcal{G}$ is $\chi$-bounded. Here we show that for every nonnegative integer $c$, $\mathcal{G}_c$ is  $\chi$-bounded, where $\mathcal{G}_c$ is the class of graphs whose every induced subgraph satisfies $\alpha \omega \geq n-c$. (Hajnal asked, and Lov\'{a}sz proved, that $\mathcal{G}_0$ is exactly the class of perfect graphs.)

\begin{theorem}
$\mathcal{G}_c$ is $\chi$-bounded. 
\end{theorem}

\begin{proof}
Let $G \in \mathcal{G}_c$. We show that $G$ does not contain $c+1$ pairwise anticomplete odd holes. Suppose  otherwise. Let $2l_1+1, \cdots , 2l_{c+1}+1$ be the hole lengths. 
Then the subgraph $H$ induced by those vertices will have $\alpha(H) = l_1 + \cdots + l_{c+1}$, $\omega(H) = 2$, and $|V(H)| = 2(l_1 + \cdots l_{c+1}) + c+1$, a contradiction to $\alpha(H)\omega(H) \geq |V(H)|-c$. 
Hence $G$ does not contain $c+1$ pairwise anticomplete odd holes. But the $\chi$-boundedness of graphs not containing $c+1$ 
pairwise anticomplete odd holes is the main result of Scott-Seymour \cite{SS7} that they use to prove  Gy\'{a}rf\'{a}s' complementation conjecture (\cite{AG}, Conjecture 6.3). 
\end{proof}

It is worth noting that the $\chi$-boundedness of the class of graphs not containing $c+1$ pairwise anticomplete odd holes implies the $\chi$-boundedness of two classes, thus resolving Conjecture 6.3 
and solving Problem 6.8 in Gy\'{a}rf\'{a}s' original paper (\cite{AG}) containing 44 problems. The relationship between Conjecture 6.3 and Problem 6.8 is not clear, and looks like neither implies the other.  

\end{section}

\begin{acknowledgements}
  I would like to thank Paul Seymour for several inspiring discussions on chi-boundedness, and Maria Chudnovsky for telling me about weakly chordal graphs.
\end{acknowledgements}


\end{document}